\documentclass[11pt]{article}

\usepackage[a4paper,margin=1.1in]{geometry}
\usepackage{amsmath,amssymb,amsthm,mathtools}
\usepackage{hyperref}
\hypersetup{hidelinks}
\usepackage{enumitem}
\usepackage{graphicx}
\graphicspath{{figures/}}

\usepackage[style=numeric,sorting=none,doi=true,url=false,isbn=true,eprint=false,backend=biber]{biblatex}
\newtheorem{theorem}{Theorem}[section]
\newtheorem{proposition}[theorem]{Proposition}
\newtheorem{lemma}[theorem]{Lemma}
\newtheorem{corollary}[theorem]{Corollary}

\theoremstyle{definition}
\newtheorem{definition}[theorem]{Definition}
\newtheorem{remark}[theorem]{Remark}

\newcommand{\C}{\mathbb C}
\newcommand{\R}{\mathbb R}
\newcommand{\dd}{\,d}
\newcommand{\Ree}{\operatorname{Re}}
\newcommand{\Imm}{\operatorname{Im}}
\newcommand{\Res}{\operatorname*{Res}}

\title{A Gaussian–Perron Prime-Side Defect and Local Profiles Near Critical-Line Zeros of the Riemann Zeta Function}
\author{Netzer Moriya}
\date{}
                
\begin{document}
\maketitle
                
\noindent
siOnet -- Applied Modeling Research, Edison, NJ, USA\\
Corresponding author: \texttt{netzer@si-o-net.com}

\begin{abstract}
We introduce a Gaussian--Perron prime-force defect that compares a smoothed
prime-side logarithmic force with the logarithmic derivative of the Riemann
zeta function. The construction turns the explicit formula into a local
diagnostic for zero geometry. Its kernel produces an error-function prime
weight and an anisotropic zero-side damping law, with an explicit boundary
separating amplified and suppressed nonlocal zero contributions. We prove an
exact zero-side formula, derive a universal selected-zero profile on the
logarithmic scale, and formulate a finite-window damping certificate for
non-selected residues. Under explicit damping, pole, and contour-regularity
hypotheses, these ingredients localize the full defect near a selected zero.
Assuming the Riemann Hypothesis and the stated pole-damping condition, the
full defect near each fixed simple critical-line zero has the selected-zero
profile up to an exponentially small nonlocal remainder. 
The framework provides a local diagnostic for zero geometry associated with
the Riemann zeta function.
\end{abstract}

\section{Introduction}
\label{sec:Introduction}

The completed Riemann \(\xi\)-function,
\begin{equation}
        \xi(s)=\frac12 s(s-1)\pi^{-s/2}\Gamma(s/2)\zeta(s),
\end{equation}
packages the functional equation and the non-trivial zeros of \(\zeta\). Its
logarithmic derivative has the horizontal-force interpretation
\begin{equation}
        \Ree\frac{\xi'}{\xi}(\sigma+it)
        =
        \partial_\sigma\log|\xi(\sigma+it)|.
\end{equation}
Thus the quantity
\(B(s)=\left(\sigma-\tfrac12\right)\Ree(\xi'/\xi)(s)\),
formally defined in Section~\ref{sec:Basic_Notation}, measures whether the
completed logarithmic force points away from the critical line.

This viewpoint is tied to horizontal monotonicity and positivity formulations
related to the Riemann Hypothesis. Lagarias \cite{Lagarias1999} studied a
positivity property of the Riemann \(\xi\)-function. Sondow--Dumitrescu
\cite{SondowDumitrescu} proved a monotonicity property of \(|\xi|\) on
horizontal half-lines and obtained a reformulation of RH.
Matiyasevich--Saidak--Zvengrowski
\cite{MatiyasevichSaidakZvengrowski} developed related monotonicity results
for \(\zeta\), \(L\)-functions, and associated functions. Singh
\cite{Singh2025} later studied a monotonicity property of Riemann's
\(\xi\)-function along certain curves.

More recently, Gold{\v{s}}tein--Grigutis \cite{GoldsteinGrigutis} and
Grigutis--Tur{\v{c}}inskas \cite{GrigutisTurcinskas2026} studied positivity of
\(\Ree(\xi'/\xi)\) in the right half of the critical strip and near the
critical line, including explicit zero-sum bounds and hypothetical off-line
configurations. The present paper develops a local theory: it builds a
Gaussian-smoothed prime-side defect and analyzes the zero geometry seen by
that defect near a selected zero.

The analytic starting point is
\begin{equation}
        \frac{\xi'}{\xi}(s)=A(s)+\frac{\zeta'}{\zeta}(s),
        \qquad
        A(s)=
        \frac1s+\frac1{s-1}
        -\frac12\log\pi
        +\frac12\psi(s/2),
        \label{eq:intro-xi-decomposition}
\end{equation}
together with
\begin{equation}
        \frac{\zeta'}{\zeta}(s)
        =
        -\sum_{n=1}^{\infty}\frac{\Lambda(n)}{n^s}
        \qquad
        (\Ree s>1).
\end{equation}
Thus \(\zeta'/\zeta\) is the prime-side component of the completed
logarithmic force. Contour shifting then relates prime sums to zero sums
through the classical explicit formula; see \cite{Titchmarsh,Edwards}.

\subsection{Objective}
\label{subsec:objective}

We construct a Gaussian--Perron smoothed prime field \(P_{X,\alpha}(s)\) and
study the defect
\begin{equation}
        \Delta_{X,\alpha}(s)
        =
        \left(\sigma-\frac12\right)
        \Ree\left(P_{X,\alpha}(s)-\frac{\zeta'}{\zeta}(s)\right).
\end{equation}
A single kernel controls both sides of the explicit formula. On the prime
side it gives an explicit error-function cutoff and rapidly decaying
prime-power tails. On the zero side it gives an anisotropic damping law that
separates a selected zero from the surrounding zero cloud.

The main results are:
\begin{enumerate}[label=\textup{(\roman*)}]
\item a Gaussian--Perron prime smoothing with an explicit error-function
prime weight;
\item an explicit zero-side formula for the corresponding prime-force defect;
\item an exact selected-zero profile on the logarithmic scale;
\item a universal bounded selected-zero profile at a critical-line zero and a
linear selected-zero response for a hypothetical off-line zero;
\item an explicit damping-region geometry and a finite-window certificate for
exponential suppression of non-selected zero contributions;
\item a conditional localization theorem under damping, pole, and
shifted-contour hypotheses;
\item an RH-conditional full local profile theorem near each fixed simple
critical-line zero satisfying the pole-damping condition.
\end{enumerate}

The analysis is local and defect-theoretic. It focuses on the selected-zero
profile and the nonlocal remainder generated by a Gaussian-smoothed
prime-side force, within the broader RH-level setting of horizontal
monotonicity and global positivity.

\subsection{Relation to prior work}
\label{subsec:relation-prior-work}

The paper combines four themes: logarithmic derivatives, explicit formulas,
smoothed prime sums, and horizontal monotonicity. The monotonicity literature
cited above studies the completed logarithmic derivative itself. In contrast,
the object here is a defect between a smoothed prime-side field and
\(\zeta'/\zeta\), weighted by horizontal distance from the critical line. The
results are therefore local asymptotic statements for this defect within the
RH-related horizontal monotonicity setting.

Smoothed prime sums and Dirichlet polynomials are standard tools in the
analysis of \(\zeta\). Soundararajan used prime-power sums in RH-conditional
moment bounds \cite{Soundararajan2009}, and
Radziwiłł--Soundararajan used short prime sums in their proof of Selberg's
central limit theorem for \(\log|\zeta(1/2+it)|\)
\cite{RadziwillSoundararajan2017}. These works provide the broader zeta
statistics context for smoothed prime-side approximations, while the present
paper studies the associated horizontal force defect.

A second related line separates prime and zero contributions through finite
Euler products and hybrid Euler--Hadamard formulas. Gonek \cite{Gonek}
studied finite Euler products and RH-level approximation phenomena.
Gonek--Hughes--Keating \cite{GonekHughesKeating} used a smoothed explicit
formula to approximate \(\zeta\) by a prime product multiplied by a product
over nearby zeros. Bui--Gonek--Milinovich \cite{BuiGonekMilinovich2015}
applied this framework to moments of \(\zeta'\) at zeros. The present work is
related in spirit but differs in object and scale: it studies a horizontal
logarithmic-force defect near one selected zero.

The contribution lies in the Gaussian--Perron prime-force defect
\(\Delta_{X,\alpha}\) and the logarithmic selected-zero scaling
\(s=\rho_0+\lambda/\log X\). The classical contour shift and smoothed
explicit formula supply the analytic mechanism. In this scale, one selected
zero produces an exact local residue profile, while the remaining zero cloud
is filtered by the anisotropic functional \(Q_\alpha\).

This differs from horizontal monotonicity criteria, which concern
\(\Re(\xi'/\xi)\), \(|\xi|\), or related completed zeta quantities directly.
It also differs from hybrid Euler--Hadamard formulas, which approximate
\(\zeta(s)\) by prime and zero factors. Here the object is a local defect
between a Gaussian-smoothed prime-side logarithmic force and
\(\zeta'/\zeta\).

Thus the smoothed prime-side force becomes a local probe of zero geometry. The
same Gaussian--Perron smoothing defines a computable prime-power field and,
after contour shifting, resolves zero contributions through an anisotropic
zero-side damping law. Section~\ref{sec:computational-diagnostics} gives a
direct prime-side numerical check of the selected-zero local profile.

\subsection{Notation and Standing Conventions}
\label{sec:Basic_Notation}

Throughout, \(s=\sigma+it\). The von Mangoldt function is \(\Lambda(n)\).
Non-trivial zeros of \(\zeta\) are denoted by \(\rho\), counted with
multiplicity \(m_\rho\). We write
\begin{equation}
        \psi(z)=\frac{\Gamma'(z)}{\Gamma(z)}
\end{equation}
for the digamma function and
\begin{equation}
        \operatorname{erfc}(x)
        =
        \frac{2}{\sqrt{\pi}}\int_x^\infty e^{-u^2}\,\dd u
\end{equation}
for the complementary error function.

We use
\begin{equation}
        L(s)=\frac{\xi'}{\xi}(s),
        \qquad
        L(s)=A(s)+\frac{\zeta'}{\zeta}(s),
        \label{eq:xi-log-decomposition}
\end{equation}
with \(A(s)\) as in Eq.~\eqref{eq:intro-xi-decomposition}. This follows by
taking the logarithmic derivative of the completed \(\xi\)-function, since
\begin{equation}
        \frac{\dd}{\dd s}\log\Gamma(s/2)=\frac12\psi(s/2).
\end{equation}
The decomposition is standard; see \cite{Edwards,Titchmarsh}.

\begin{definition}[Horizontal $\xi$-force]
At a point where \(\xi(s)\neq0\), we define
\begin{equation}
        B(s)=
        \left(\sigma-\frac12\right)
        \Ree\frac{\xi'}{\xi}(s).
\end{equation}
\end{definition}

\begin{remark}
Since
\begin{equation}
        \Ree\frac{\xi'}{\xi}(s)=\partial_\sigma\log|\xi(s)|,
\end{equation}
the sign of \(B(s)\) records the horizontal direction of the completed
logarithmic force. Electrostatic interpretations of logarithmic derivatives
associated with the zeta function have also appeared in the literature; see,
for example, \cite{LeClair}.
\end{remark}

\begin{remark}[Symmetry across the critical line]
\label{rem:symmetry}
The functional equation \(\xi(s)=\xi(1-s)\) gives \(L(1-s)=-L(s)\), and
\(L(\overline{s})=\overline{L(s)}\). Hence, for \(s=1/2+x+it\),
\begin{equation}
        \Ree L\left(\frac12-x+it\right)
        =
        -\Ree L\left(\frac12+x+it\right),
\end{equation}
so \(\Ree(\xi'/\xi)\) is odd across the critical line, while \(B(s)\) is even
under this reflection. The present contribution builds on this standard
symmetry through the prime-side defect \(\Delta_{X,\alpha}\), its
selected-zero profile, the conditional nonlocal localization theorem, and the
RH-conditional full local profile theorem.
\end{remark}

\section{A Gaussian--Perron Prime Smoothing}
\label{sec:A_Gaussian_Perron_Prime_Smoothing}

Let
\begin{equation}
        X>1,
        \qquad
        Y=\log X,
        \qquad
        \alpha>0.
\end{equation}
We use the Gaussian--Perron kernel
\begin{equation}
        H_{X,\alpha}(z)=X^z\exp\!\left(\alpha^2Yz^2\right).
        \label{eq:kernel}
\end{equation}
The positive quadratic sign is essential: for \(z=c+iv\),
\begin{equation}
        \left|\exp\!\left(\alpha^2Yz^2\right)\right|
        =
        \exp\!\left(\alpha^2Y(c^2-v^2)\right),
\end{equation}
which decays rapidly on vertical lines, whereas \(\exp(-\alpha^2Yz^2)\)
would grow as \(|v|\to\infty\).

\begin{definition}[Gaussian--Perron smoothed prime field]
Let \(s=\sigma+it\), and choose \(c>\max\{0,1-\sigma\}\). Define
\begin{equation}
        P_{X,\alpha}(s)
        =
        \frac{1}{2\pi i}
        \int_{(c)}
        \frac{\zeta'}{\zeta}(s+z)
        \frac{X^z\exp(\alpha^2Yz^2)}{z}\,\dd z.
        \label{eq:PXalpha-definition}
\end{equation}
\end{definition}

\begin{proposition}[Prime-side representation]
For any \(c>\max\{0,1-\sigma\}\),
\begin{equation}
        P_{X,\alpha}(s)
        =
        -\sum_{n=1}^{\infty}\frac{\Lambda(n)}{n^s}W_{X,\alpha}(n),
\end{equation}
where
\begin{equation}
        W_{X,\alpha}(n)
        =
        \frac{1}{2\pi i}
        \int_{(c)}
        \frac{X^z\exp(\alpha^2Yz^2)}{z}n^{-z}\,\dd z.
\end{equation}
Moreover,
\begin{equation}
        W_{X,\alpha}(n)
        =
        \frac12\operatorname{erfc}
        \left(
        \frac{\log n-\log X}{2\alpha\sqrt{Y}}
        \right).
        \label{eq:erfc-weight}
\end{equation}
Thus \(W_{X,\alpha}\) is a Gaussian-smoothed Perron cutoff centered at
\(n=X\), with logarithmic width \(\sqrt{\log X}\). Such weights are related
in spirit to weighted finite Euler products and hybrid Euler--Hadamard
factorizations \cite{Gonek,GonekHughesKeating}; Eq.~\eqref{eq:erfc-weight} is the
Perron-weight form used here.
\end{proposition}

\begin{proof}
On \(z=c+iv\), one has \(\Ree(s+z)>1\), so
\begin{equation}
        \frac{\zeta'}{\zeta}(s+z)
        =
        -\sum_{n=1}^{\infty}\Lambda(n)n^{-s-z}.
\end{equation}
The interchange with the integral in Eq.~\eqref{eq:PXalpha-definition} is
justified by
\begin{equation}
        \left|
        \frac{X^z\exp(\alpha^2Yz^2)}{z}
        \right|
        =
        \frac{X^c\exp(\alpha^2Yc^2)}{|c+iv|}
        \exp(-\alpha^2Yv^2)
\end{equation}
and by the convergence of
\begin{equation}
        \sum_{n=1}^{\infty}\Lambda(n)n^{-\sigma-c}
        \int_{-\infty}^{\infty}
        \frac{X^c\exp(\alpha^2Yc^2)}{|c+iv|}
        \exp(-\alpha^2Yv^2)\,\dd v,
\end{equation}
which follows from \(\sigma+c>1\) and \(c>0\). Hence the stated series
representation follows.

It remains to compute the weight. Put \(u=\log n-\log X\) and
\begin{equation}
        \Phi(u)=
        \frac{1}{2\pi i}
        \int_{(c)}
        \frac{\exp(\alpha^2Yz^2-uz)}{z}\,\dd z.
\end{equation}
Differentiation under the integral sign is justified locally uniformly in
\(u\) by Gaussian domination, giving
\begin{equation}
        \Phi'(u)
        =
        -\frac{1}{2\pi i}
        \int_{(c)}
        \exp(\alpha^2Yz^2-uz)\,\dd z.
\end{equation}
Since
\begin{equation}
        \alpha^2Yz^2-uz
        =
        \alpha^2Y
        \left(z-\frac{u}{2\alpha^2Y}\right)^2
        -
        \frac{u^2}{4\alpha^2Y},
\end{equation}
the entire integrand may be shifted to \(\Ree z=u/(2\alpha^2Y)\); the
horizontal sides vanish because, on \(z=r+iT\),
\begin{equation}
        \left|\exp(\alpha^2Yz^2-uz)\right|
        =
        \exp\left(\alpha^2Y(r^2-T^2)-ur\right).
\end{equation}
Thus
\begin{equation}
\begin{aligned}
        \frac{1}{2\pi i}
        \int_{(c)}
        \exp(\alpha^2Yz^2-uz)\,\dd z
        &=
        \frac{e^{-u^2/(4\alpha^2Y)}}{2\pi}
        \int_{-\infty}^{\infty}e^{-\alpha^2Yv^2}\,\dd v        \\
        &=
        \frac{1}{2\sqrt{\pi}\alpha\sqrt Y}
        \exp\!\left(-\frac{u^2}{4\alpha^2Y}\right),
\end{aligned}
\end{equation}
and therefore
\begin{equation}
        \Phi'(u)
        =
        -\frac{1}{2\sqrt{\pi}\alpha\sqrt Y}
        \exp\!\left(-\frac{u^2}{4\alpha^2Y}\right).
\end{equation}
Since \(c>0\), the original integral gives
\begin{equation}
        |\Phi(u)|
        \le
        C_{c,\alpha,Y}e^{-cu},
        \qquad u\to+\infty.
\end{equation}
Thus \(\Phi(u)\to0\) as \(u\to+\infty\), and integration of \(\Phi'\) from
\(u\) to \(+\infty\) yields
\begin{equation}
        \Phi(u)
        =
        \frac12\operatorname{erfc}
        \left(\frac{u}{2\alpha\sqrt Y}\right).
\end{equation}
Substituting \(u=\log n-\log X\) proves Eq.~\eqref{eq:erfc-weight}.
\end{proof}

We shall use the following definition throughout the sequel.

\begin{definition}[Prime-force defect]
At points where \(\zeta'/\zeta\) is finite, equivalently where \(s\) is
neither a zero nor the pole of \(\zeta\), define
\begin{equation}
        \Delta_{X,\alpha}(s)
        =
        \left(\sigma-\frac12\right)
        \Ree\left(P_{X,\alpha}(s)-\frac{\zeta'}{\zeta}(s)\right).
        \label{eq:defect-definition}
\end{equation}
\end{definition}

\section{Explicit Formula for the Defect}

\begin{lemma}[Logarithmic derivative away from zeros]
\label{lem:log-derivative-away}
Let \(A<B\) be fixed. Suppose \(w=u+it\), \(A\le u\le B\), is separated
from the zeros and pole of \(\zeta\) by distance at least
\(c_0/\log(3+|t|)\), where \(c_0>0\) is fixed. Then
\begin{equation}
        \frac{\zeta'}{\zeta}(w)
        \ll_{A,B,c_0}\log^2(3+|t|).
        \label{eq:log-derivative-away}
\end{equation}
\end{lemma}

\begin{proof}
This is the standard logarithmic-derivative estimate for \(\zeta\) away from
its zeros and pole, obtained from the partial-fraction expansion for
\(\zeta'/\zeta\) together with the Riemann--von Mangoldt zero-counting
estimate. We apply it in fixed vertical strips at points satisfying the
displayed separation condition; see \cite{Titchmarsh,MontgomeryVaughan}.
\end{proof}

\begin{lemma}[Admissible horizontal heights]
\label{lem:admissible-heights}
Let \(A<B\) and \(\tau\in\R\) be fixed. There exists a sequence
\(T_j\to\infty\) such that the horizontal segments
\begin{equation}
        \{u+i(\tau\pm T_j):A\le u\le B\}
\end{equation}
avoid the zeros and pole of \(\zeta\), and
\begin{equation}
        \frac{\zeta'}{\zeta}(u+i(\tau\pm T_j))
        \ll_{A,B,\tau}(\log T_j)^2
\end{equation}
uniformly for \(A\le u\le B\). The sequence may be chosen so that these
segments remain at distance \(\gg1/\log T_j\) from the ordinates of the
non-trivial zeros.
\end{lemma}

\begin{proof}
By the Riemann--von Mangoldt formula, the number of zero ordinates in
\([T,T+1]\) is \(O(\log T)\). Hence one can choose \(T_j\to\infty\) so that
the heights \(\tau\pm T_j\) are at distance \(\gg1/\log T_j\) from all zero
ordinates in the relevant bounded horizontal strip, and also avoid the pole.
Applying Lemma~\ref{lem:log-derivative-away} on the two horizontal segments
gives the stated uniform bound.
\end{proof}

We can now state the explicit prime-force defect formula.

\begin{theorem}[Explicit prime-force defect formula]
\label{thm:explicit-defect}
Let \(s\) be neither a zero nor the pole \(s=1\) of \(\zeta\). Choose
\(c>0\) with \(\Ree(s+c)>1\), and choose \(d>0\) such that no pole of the
integrand lies on \(\Ree z=-d\). (On \(\Ree z=c\) the integrand has no poles,
since \(\Ree(s+z)>1\) there and \(z\neq0\).) Then
\begin{equation}
\begin{aligned}
        P_{X,\alpha}(s)-\frac{\zeta'}{\zeta}(s)
        &=
        \sum_{\rho:\,-d<\Ree(\rho-s)<c}
        \frac{m_\rho X^{\rho-s}\exp(\alpha^2Y(\rho-s)^2)}{\rho-s}        \\
        &\quad
        -
        \mathbf 1_{\{-d<\Ree(1-s)<c\}}
        \frac{X^{1-s}\exp(\alpha^2Y(1-s)^2)}{1-s}                    \\
        &\quad
        +
        \sum_{m\ge1:\,-d<\Ree(-2m-s)<c}
        \frac{X^{-2m-s}\exp(\alpha^2Y(-2m-s)^2)}{-2m-s}              \\
        &\quad+
        R_{X,\alpha,d}(s),
\end{aligned}
        \label{eq:explicit-field-defect}
\end{equation}
where
\begin{equation}
        R_{X,\alpha,d}(s)
        =
        \lim_{j\to\infty}
        \frac{1}{2\pi i}
        \int_{-d-iT_j}^{-d+iT_j}
        \frac{\zeta'}{\zeta}(s+z)
        \frac{X^z\exp(\alpha^2Yz^2)}{z}\,\dd z.
        \label{eq:shifted-contour-limit}
\end{equation}
Here \(T_j\) is an admissible-height sequence from
Lemma~\ref{lem:admissible-heights}, applied to the crossed strip. If the
shifted-line integral is absolutely convergent, this limit is the ordinary
vertical integral over \(\Ree z=-d\). For the explicit formula, \(d>0\) is
otherwise arbitrary, apart from avoiding poles on the shifted contour; later
localization estimates impose additional restrictions such as
\(d<1/\alpha^2\).

Consequently,
\begin{equation}
        \Delta_{X,\alpha}(s)
        =
        \left(\sigma-\frac12\right)
        \Ree
        \left[
        \mathcal Z_{X,\alpha}(s)
        +
        \mathcal P_{X,\alpha}(s)
        +
        \mathcal T_{X,\alpha}(s)
        +
        R_{X,\alpha,d}(s)
        \right],
        \label{eq:explicit-defect}
\end{equation}
where the three terms are the zero, pole, and trivial-zero terms displayed in
Eq.~\eqref{eq:explicit-field-defect}; explicitly,
\begin{equation}
        \mathcal Z_{X,\alpha}(s)
        =
        \sum_{\rho:\,-d<\Ree(\rho-s)<c}
        \frac{m_\rho X^{\rho-s}\exp(\alpha^2Y(\rho-s)^2)}{\rho-s},
\end{equation}
\begin{equation}
        \mathcal P_{X,\alpha}(s)
        =
        -
        \mathbf 1_{\{-d<\Ree(1-s)<c\}}
        \frac{X^{1-s}\exp(\alpha^2Y(1-s)^2)}{1-s},
\end{equation}
and
\begin{equation}
        \mathcal T_{X,\alpha}(s)
        =
        \sum_{m\ge1:\,-d<\Ree(-2m-s)<c}
        \frac{X^{-2m-s}\exp(\alpha^2Y(-2m-s)^2)}{-2m-s}.
\end{equation}
\end{theorem}

\begin{proof}
Set
\begin{equation}
        F_s(z)=
        \frac{\zeta'}{\zeta}(s+z)
        \frac{X^z\exp(\alpha^2Yz^2)}{z}.
\end{equation}
Shift the contour from \(\Ree z=c\) to \(\Ree z=-d\), keeping upward
orientation on both vertical lines. Equivalently, the left side of the
positively oriented rectangle has been moved to the right-hand side of the
identity and rewritten with upward orientation. The finite-rectangle residue
theorem gives
\begin{equation}
\begin{aligned}
        \frac{1}{2\pi i}
        \int_{c-iT_j}^{c+iT_j}F_s(z)\,\dd z
        &=
        \frac{1}{2\pi i}
        \int_{-d-iT_j}^{-d+iT_j}F_s(z)\,\dd z             \\
        &\quad+
        \sum_{\substack{z_*:\,-d<\Ree z_*<c\\ |\Imm z_*|<T_j}}
        \Res_{z=z_*}F_s(z)+E_j,
\end{aligned}
\end{equation}
where \(E_j\) is the signed contribution of the horizontal sides. The heights
are chosen by Lemma~\ref{lem:admissible-heights}, with
\(A=\Ree s-d\), \(B=\Ree s+c\), and \(\tau=\Imm s\).

On a horizontal side \(z=u\pm iT_j\),
\begin{equation}
        \left|\exp(\alpha^2Yz^2)\right|
        =
        \exp\!\left(\alpha^2Y(u^2-T_j^2)\right),
        \qquad
        \left|\frac{X^z}{z}\right|
        \le
        \frac{X^c}{T_j}.
\end{equation}
Together with Lemma~\ref{lem:admissible-heights}, this gives
\begin{equation}
        E_j=
        O_{s,c,d,X,\alpha,Y}
        \left(
        \frac{(\log T_j)^2}{T_j}
        e^{-\alpha^2YT_j^2}
        \right)=o(1)
        \qquad (j\to\infty).
\end{equation}
The right vertical integrals converge to \(P_{X,\alpha}(s)\), and the left
vertical integrals converge to \(R_{X,\alpha,d}(s)\) by
Eq.~\eqref{eq:shifted-contour-limit}.

We next pass to the infinite zero sum. For a crossed non-trivial zero
\(\rho=\beta+i\gamma\), write
\begin{equation}
        \rho-s=x+i(\gamma-\Imm s),
        \qquad -d<x<c .
\end{equation}
Then
\begin{equation}
        \left|
        X^{\rho-s}\exp(\alpha^2Y(\rho-s)^2)
        \right|
        \le
        C_{c,d,X,\alpha,Y}
        \exp\!\left(-\alpha^2Y(\gamma-\Imm s)^2\right).
\end{equation}
Since \(s\) is not a zero, the denominator is bounded away from zero in any
bounded ordinate interval; for large \(|\gamma-\Imm s|\),
\begin{equation}
        |\rho-s|\gg_{s,c,d}1+|\gamma-\Imm s|.
\end{equation}
The Riemann--von Mangoldt estimate, applied in unit intervals, therefore gives
absolute convergence of
\begin{equation}
        \sum_{\rho:\,-d<\Ree(\rho-s)<c}
        \frac{
        m_\rho
        \left|X^{\rho-s}\exp(\alpha^2Y(\rho-s)^2)\right|
        }{
        |\rho-s|}
        <\infty .
\end{equation}
Thus the finite residue sums converge to the displayed zero sum. The
trivial-zero sum is finite for fixed \(s,c,d\), and the pole term is a single
residue.

The residue at \(z=0\) is \(\zeta'/\zeta(s)\). A zero \(s+z=\rho\) of
multiplicity \(m_\rho\) contributes
\begin{equation}
        \frac{m_\rho X^{\rho-s}\exp(\alpha^2Y(\rho-s)^2)}{\rho-s}.
\end{equation}
The pole \(s+z=1\) contributes
\begin{equation}
        -\frac{X^{1-s}\exp(\alpha^2Y(1-s)^2)}{1-s},
\end{equation}
and a trivial zero \(s+z=-2m\) contributes
\begin{equation}
        \frac{X^{-2m-s}\exp(\alpha^2Y(-2m-s)^2)}{-2m-s}.
\end{equation}
Subtracting the residue at \(z=0\) gives Eq.~\eqref{eq:explicit-field-defect};
multiplication by \(\sigma-1/2\) and taking real parts gives
Eq.~\eqref{eq:explicit-defect}.
\end{proof}
Concerning the preceding result, the new object is the
horizontal prime-force defect \(\Delta_{X,\alpha}\), together with the local
profiles extracted from it; the contour shift supplies the classical
explicit-formula mechanism.

\section{Trivial Zero Contribution}

\begin{proposition}[Trivial zero contribution]
\label{prop:trivial-zero}
Let \(S\) be a fixed compact subset of the critical strip and set
\begin{equation}
        \sigma_- = \inf_{w\in S}\Ree w .
\end{equation}
For fixed \(d>0\) and \(s\in S\),
\begin{equation}
        \mathcal T_{X,\alpha}(s)
        =
        \sum_{m\ge1:\,-d<\Ree(-2m-s)<c}
        \frac{X^{-2m-s}\exp(\alpha^2Y(-2m-s)^2)}{-2m-s}
\end{equation}
has finitely many terms. If \(d<2+\sigma_-\), then
\(\mathcal T_{X,\alpha}(s)=0\) for all \(s\in S\).
\end{proposition}

\begin{proof}
A term occurs precisely when \(-d<\Ree(-2m-s)<c\). Since
\(\Ree(-2m-s)=-2m-\Ree s\) and \(s\in S\), the crossed set contains
finitely many \(m\ge1\). If \(d<2+\sigma_-\), then
\(\Ree(-2m-s)\le -2-\sigma_-<-d\), so all trivial-zero residues remain
to the left of the shifted contour.
\end{proof}

\section{Single-Zero Profiles}

Fix a simple non-trivial zero \(\rho_0=\beta_0+i\gamma_0\) and set
\begin{equation}
        s=\rho_0+\frac{\lambda}{Y},
        \qquad
        \lambda=a+ib,
        \qquad
        Y=\log X.
\end{equation}
Uniform asymptotics are for \(\lambda\in K\), where
\(K\Subset\C\setminus\{0\}\), unless stated otherwise.

\begin{definition}[Single-zero local contribution]
For \(s\neq\rho_0\), define
\begin{equation}
        \mathcal L_{\rho_0,X,\alpha}(s)
        =
        \left(\sigma-\frac12\right)
        \Ree\left(
        \frac{X^{\rho_0-s}\exp(\alpha^2Y(\rho_0-s)^2)}{\rho_0-s}
        \right).
\end{equation}
\end{definition}

The following lemma gives an exact evaluation of the selected residue appearing
in the zero-side explicit formula.

\begin{lemma}[Exact selected-zero formula]
\label{lem:exact-local}
Let \(\rho_0=\beta_0+i\gamma_0\) be simple and let
\begin{equation}
        s=\rho_0+\frac{\lambda}{Y},
        \qquad
        \lambda=a+ib\neq0.
\end{equation}
Then
\begin{equation}
        \mathcal L_{\rho_0,X,\alpha}
        \left(\rho_0+\frac{\lambda}{Y}\right)
        =
        -
        \left(\beta_0-\frac12+\frac{a}{Y}\right)
        Y\,
        \Ree\left(
        \frac{\exp(-\lambda+\alpha^2\lambda^2/Y)}{\lambda}
        \right).
        \label{eq:exact-local}
\end{equation}
\end{lemma}

\begin{proof}
Here \(\rho_0-s=-\lambda/Y\), hence
\(X^{\rho_0-s}=e^{-\lambda}\),
\(\exp(\alpha^2Y(\rho_0-s)^2)=\exp(\alpha^2\lambda^2/Y)\), and
\((\rho_0-s)^{-1}=-Y/\lambda\). Also
\(\sigma-1/2=\beta_0-1/2+a/Y\). Substitution proves the formula.
\end{proof}

We now turn to the critical-line selected-zero profile.

\begin{theorem}[Universal bounded selected-zero profile at a critical-line zero]
\label{thm:central-profile}
Let \(\rho_0=1/2+i\gamma_0\) be simple. For
\begin{equation}
        s=\rho_0+\frac{\lambda}{Y},
        \qquad
        \lambda=a+ib\neq0,
\end{equation}
\begin{equation}
        \mathcal L_{\rho_0,X,\alpha}(s)
        =
        -a\,
        \Ree\left(
        \frac{\exp(-\lambda+\alpha^2\lambda^2/Y)}{\lambda}
        \right).
        \label{eq:central-profile-exact}
\end{equation}
Consequently, as \(X\to\infty\), uniformly for \(\lambda\in K\),
\begin{equation}
        \mathcal L_{\rho_0,X,\alpha}(s)
        =
        -a\,
        \Ree\left(\frac{e^{-\lambda}}{\lambda}\right)
        +O_K\left(\frac1Y\right).
        \label{eq:central-profile}
\end{equation}
\end{theorem}

\begin{proof}
In Eq.~\eqref{eq:exact-local}, \(\beta_0=1/2\), so
\(\beta_0-1/2+a/Y=a/Y\), giving Eq.~\eqref{eq:central-profile-exact}. The
uniform expansion \(\exp(\alpha^2\lambda^2/Y)=1+O_K(1/Y)\) gives
Eq.~\eqref{eq:central-profile}.
\end{proof}

Concerning the preceding result, we note that
the leading term in Eq.~\eqref{eq:central-profile} is independent of \(\alpha\).
The Gaussian width affects the \(O_K(1/Y)\) correction and the nonlocal
damping, while the leading critical-line selected-zero profile remains
universal.

We now state the following theorem.

\begin{theorem}[Linear selected-zero spike at a hypothetical off-line zero]
\label{thm:offline-spike}
Let \(\rho_0=\beta_0+i\gamma_0\) be simple with \(\beta_0\neq1/2\). For
\begin{equation}
        s=\rho_0+\frac{\lambda}{Y},
        \qquad
        \lambda=a+ib\neq0,
\end{equation}
\begin{equation}
        \frac1Y\mathcal L_{\rho_0,X,\alpha}(s)
        =
        -
        \left(\beta_0-\frac12+\frac{a}{Y}\right)
        \Ree\left(
        \frac{\exp(-\lambda+\alpha^2\lambda^2/Y)}{\lambda}
        \right).
        \label{eq:offline-spike}
\end{equation}
Consequently, as \(X\to\infty\), uniformly for \(\lambda\in K\),
\begin{equation}
        \frac1Y\mathcal L_{\rho_0,X,\alpha}(s)
        \longrightarrow
        -
        \left(\beta_0-\frac12\right)
        \Ree\left(\frac{e^{-\lambda}}{\lambda}\right).
        \label{eq:offline-spike-limit}
\end{equation}
\end{theorem}

\begin{proof}
Eq.~\eqref{eq:offline-spike} is Eq.~\eqref{eq:exact-local} divided by
\(Y\). Since \(K\Subset\C\setminus\{0\}\), \(e^{-\lambda}/\lambda\) is
bounded on \(K\), while \(a/Y\to0\) and
\(\exp(\alpha^2\lambda^2/Y)\to1\) uniformly. This proves the limit.
\end{proof}

\begin{remark}[Scope of Theorem~\ref{thm:offline-spike}]
\label{rem:offline-scope}
Theorem~\ref{thm:offline-spike} concerns the isolated selected residue
\(\mathcal L_{\rho_0,X,\alpha}\) only. The nonlocal remainder
\(\mathcal E_{\rho_0,X,\alpha}\) is not controlled at a hypothetical off-line
\(\rho_0\): both Theorems~\ref{thm:full-remainder-localization}
and~\ref{thm:rh-full-local-profile} take \(\rho_0=1/2+i\gamma_0\), and their
pole-damping and shifted-contour hypotheses are tuned to the critical-line
case. The linear scaling in \(\log X\) predicted here is therefore a
statement about the residue term of the explicit formula, not about the full
defect \(\Delta_{X,\alpha}\). Its use as a numerical diagnostic requires a
model in which the nonlocal remainder is separately controlled.
\end{remark}

\section{Full Nonlocal Remainder}

Let \(s=\rho_0+\lambda/Y\) and write
\begin{equation}
        \Delta_{X,\alpha}(s)
        =
        \mathcal L_{\rho_0,X,\alpha}(s)
        +
        \mathcal E_{\rho_0,X,\alpha}(s),
        \label{eq:full-decomposition}
\end{equation}
where \(\mathcal E_{\rho_0,X,\alpha}\) contains all terms other than
\(\rho_0\).

\begin{proposition}[Explicit nonlocal remainder]
\label{prop:explicit-remainder}
Assume that \(\rho_0\) is a simple zero included among the zeros crossed in
Theorem~\ref{thm:explicit-defect}. Use the same \(c,d\), and the same
crossed-zero convention, as in that theorem. Then
\begin{equation}
\begin{aligned}
        \mathcal E_{\rho_0,X,\alpha}(s)
        &=
        \left(\sigma-\frac12\right)
        \Ree
        \sum_{\substack{\rho'\neq\rho_0\\ -d<\Ree(\rho'-s)<c}}
        \frac{m_{\rho'}X^{\rho'-s}\exp(\alpha^2Y(\rho'-s)^2)}{\rho'-s} \\
        &\quad
        -
        \left(\sigma-\frac12\right)
        \Ree\left[
        \mathbf 1_{\{-d<\Ree(1-s)<c\}}
        \frac{X^{1-s}\exp(\alpha^2Y(1-s)^2)}{1-s}
        \right] \\
        &\quad
        +
        \left(\sigma-\frac12\right)
        \Ree
        \sum_{m\ge1:\,-d<\Ree(-2m-s)<c}
        \frac{X^{-2m-s}\exp(\alpha^2Y(-2m-s)^2)}{-2m-s} \\
        &\quad
        +
        \left(\sigma-\frac12\right)
        \Ree R_{X,\alpha,d}(s).
\end{aligned}
        \label{eq:nonlocal-remainder}
\end{equation}
\end{proposition}

\begin{proof}
Subtract the selected \(\rho_0\)-term from the crossed non-trivial-zero sum
in Theorem~\ref{thm:explicit-defect}. The pole term, the trivial-zero term,
and the shifted-contour term remain. Since \(\rho_0\) is simple, its crossed
contribution is \(\mathcal L_{\rho_0,X,\alpha}(s)\). Taking real parts and
multiplying by \(\sigma-1/2\) gives Eq.~\eqref{eq:nonlocal-remainder}.
\end{proof}

\section{Remainder Localization Under Conditional Damping and Contour Regularity}

The estimate
\begin{equation}
        \mathcal E_{\rho_0,X,\alpha}
        \left(\rho_0+\frac{\lambda}{Y}\right)=o(1)
\end{equation}
requires separate control of crossed non-selected zeros, the pole term,
possible trivial-zero terms, and the shifted contour integral.
Here and below, \(o(1)\) in the local scale means a quantity tending to zero
as \(X\to\infty\), uniformly for \(\lambda\) in each fixed compact set
\(K\Subset\mathbb C\setminus\{0\}\).
We begin with the zero side: the Gaussian damping functional describes the
geometry of non-selected zero contributions, and the finite-window certificate
turns this geometry into an exponential bound.
The localization theorem then combines this zero damping with pole,
trivial-zero, and shifted-contour controls.

\begin{definition}[Gaussian damping functional]
Let \(\rho_0=1/2+i\gamma_0\) be fixed. For a zero \(\rho'\), define
\begin{equation}
        Q_\alpha(\rho';\rho_0)
        =
        \Ree(\rho'-\rho_0)
        +
        \alpha^2\Ree\left((\rho'-\rho_0)^2\right).
\end{equation}
\end{definition}

\begin{proposition}[Geometry of the Gaussian damping functional]
\label{prop:damping-geometry}
Write \(\rho'-\rho_0=x+iy\). Then
\begin{equation}
        Q_\alpha(\rho';\rho_0)
        =
        x+\alpha^2(x^2-y^2).
        \label{eq:Q-geometry}
\end{equation}
Thus \(Q_\alpha(\rho';\rho_0)<0\) if and only if
\begin{equation}
        y^2>x^2+\frac{x}{\alpha^2}.
        \label{eq:damping-region}
\end{equation}
The boundary \(Q_\alpha=0\) is the real curve
\begin{equation}
        y^2=x^2+\frac{x}{\alpha^2},
        \label{eq:damping-boundary}
\end{equation}
on the portion where the right-hand side is non-negative. If
\(-1/\alpha^2<x<0\), then \(Q_\alpha<0\) for every real \(y\). In particular,
if \(x=0\) and \(y\neq0\), then
\begin{equation}
        Q_\alpha(\rho';\rho_0)=-\alpha^2y^2<0.
\end{equation}
\end{proposition}

\begin{proof}
Since \(\Ree(x+iy)=x\) and \(\Ree((x+iy)^2)=x^2-y^2\),
substitution gives Eq.~\eqref{eq:Q-geometry}. The inequality \(Q_\alpha<0\) is
\(x+\alpha^2x^2-\alpha^2y^2<0\), which is equivalent to
Eq.~\eqref{eq:damping-region}. The remaining statements follow immediately.
\end{proof}

The following Proposition is used in the Remainder
Localization Under Conditional Damping and Contour Regularity argument.

\begin{proposition}[Finite-window damping certificate]
\label{prop:finite-window-certificate}
Let \(K\Subset\C\setminus\{0\}\), let \(s=\rho_0+\lambda/Y\) with
\(\lambda\in K\), and let \(\mathcal Z_M(\rho_0)\) be a finite set of
non-selected zeros \(\rho'\neq\rho_0\). Suppose
\begin{equation}
        \eta_{\alpha,M}(\rho_0)
        :=
        -
        \max_{\rho'\in\mathcal Z_M(\rho_0)}
        Q_\alpha(\rho';\rho_0)
        >0.
        \label{eq:finite-window-margin}
\end{equation}
Then, as \(Y\to\infty\),
\begin{equation}
        \sum_{\rho'\in\mathcal Z_M(\rho_0)}
        \left|
        \frac{
        m_{\rho'}X^{\rho'-s}
        \exp(\alpha^2Y(\rho'-s)^2)}
        {\rho'-s}
        \right|
        =
        O_{K,\rho_0,\alpha,\mathcal Z_M}
        \left(e^{-\eta_{\alpha,M}(\rho_0)Y/2}\right).
        \label{eq:finite-window-certificate}
\end{equation}
If \(\mathcal Z_M(\rho_0)\) is empty, the assertion is vacuous.
\end{proposition}

\begin{proof}
Assume \(\mathcal Z_M(\rho_0)\neq\emptyset\), and put
\(\delta_{\rho'}=\rho'-\rho_0\). Since
\(\rho'-s=\delta_{\rho'}-\lambda/Y\), the logarithm of the modulus of the
exponential factor equals
\begin{equation}
\begin{aligned}
&Y\Ree\left(\delta_{\rho'}-\frac{\lambda}{Y}\right)
+
\alpha^2Y\Ree\left(
\left(\delta_{\rho'}-\frac{\lambda}{Y}\right)^2
\right)       \\
&\qquad =
YQ_\alpha(\rho';\rho_0)
-\Ree\lambda
-2\alpha^2\Ree(\delta_{\rho'}\lambda)
+
\frac{\alpha^2}{Y}\Ree(\lambda^2).
\end{aligned}
\end{equation}
The last three terms are \(O_{K,\alpha,\mathcal Z_M}(1)\), uniformly for
\(\lambda\in K\). By Eq.~\eqref{eq:finite-window-margin}, the exponent is at most
\(-\eta_{\alpha,M}(\rho_0)Y+O_{K,\alpha,\mathcal Z_M}(1)\), hence at most
\(-\eta_{\alpha,M}(\rho_0)Y/2\) for large \(Y\).

Since \(\mathcal Z_M(\rho_0)\) is finite and excludes \(\rho_0\), there is
\(\mu>0\) such that \(|\rho'-\rho_0|\ge\mu\) for all
\(\rho'\in\mathcal Z_M(\rho_0)\). For large \(Y\), uniformly in
\(\lambda\in K\),
\begin{equation}
        |\rho'-s|
        =
        \left|\rho'-\rho_0-\frac{\lambda}{Y}\right|
        \ge
        \frac{\mu}{2}.
\end{equation}
Summing over the finite set proves Eq.~\eqref{eq:finite-window-certificate}.
\end{proof}

The zero-counting and large-height Gaussian-tail estimates used in the
localization theorem are recorded in
Appendix~\ref{app:conditional-localization}.

We can now state the conditional full nonlocal remainder localization theorem.

\begin{theorem}[Conditional full nonlocal remainder localization]
\label{thm:full-remainder-localization}
Let
\begin{equation}
        \rho_0=\frac12+i\gamma_0,
        \qquad
        \gamma_0>0,
\end{equation}
be a simple zero of \(\zeta\), and let \(K\Subset\C\setminus\{0\}\). Put
\begin{equation}
        s=\rho_0+\frac{\lambda}{Y},
        \qquad
        \lambda=a+ib\in K,
        \qquad
        Y=\log X.
\end{equation}
Choose \(c>1/2\). For all sufficiently large \(Y\), this ensures
\(\Ree(s+c)>1\) uniformly for \(\lambda\in K\). Choose \(d>0\) for the
explicit formula, with \(d\) satisfying the restrictions below.

Assume the following hypotheses.

\begin{enumerate}[label=\textup{(\roman*)}]
\item There exist constants \(M\ge M_0(\alpha,K)\) and \(\kappa>0\),
where \(M_0(\alpha,K)\) is as in Lemma~\ref{lem:large-height-tail}, such
that every crossed non-trivial zero \(\rho'\neq\rho_0\) with
\begin{equation}
        |\gamma'-\gamma_0|\le M
\end{equation}
satisfies
\begin{equation}
        Q_\alpha(\rho';\rho_0)\le -\kappa.
\end{equation}
Thus the bounded-height crossed zeros satisfy the finite-window damping
certificate of Proposition~\ref{prop:finite-window-certificate}.

\item The pole term is exponentially damped: there exists
\(\eta_{\rm pole}>0\) such that, uniformly for \(\lambda\in K\) and all
sufficiently large \(Y\),
\begin{equation}
        \Ree(1-s)+\alpha^2\Ree\left((1-s)^2\right)
        \le -\eta_{\rm pole}.
        \label{eq:pole-damping-condition}
\end{equation}
A convenient sufficient condition is
\begin{equation}
        \frac12+\alpha^2\left(\frac14-\gamma_0^2\right)<0.
        \label{eq:sufficient-pole-damping}
\end{equation}
Indeed, since \(s=1/2+i\gamma_0+\lambda/Y\),
\begin{equation}
        \Ree(1-s)+\alpha^2\Ree\left((1-s)^2\right)
        =
        \frac12+\alpha^2\left(\frac14-\gamma_0^2\right)+O_K(1/Y),
\end{equation}
uniformly for \(\lambda\in K\).

\item The explicit-formula contour parameters satisfy
\begin{equation}
        c>\frac12,
        \qquad
        0<d<\frac1{\alpha^2}.
\end{equation}
Moreover, uniformly for \(\lambda\in K\) and all sufficiently large \(Y\),
\begin{equation}
        d<2+\Ree\left(\rho_0+\frac{\lambda}{Y}\right).
\end{equation}
Both contour lines \(\Ree z=c\) and \(\Ree z=-d\) avoid the poles of
the integrand.

\item The shifted contour is chosen to avoid every zero and pole of
\(\zeta'/\zeta\) on \(w=s-d+iv\). On this line, assume that for some
fixed \(A>0\),
\begin{equation}
        \frac{\zeta'}{\zeta}(w)
        \ll_{d,K,\rho_0,A} \log^A(2+|v|+\gamma_0).
\end{equation}
This contour-regularity hypothesis controls the shifted line inside the
critical strip. The precise logarithmic power is
irrelevant for the theorem, because the shifted contour is multiplied by the
Gaussian factor \(e^{-\alpha^2Yv^2}\).
\end{enumerate}

Then there exist constants \(C_K,c_K>0\), depending on
\(K,\alpha,d,\kappa,M,\eta_{\rm pole},\rho_0,A\), such that
\begin{equation}
        \mathcal E_{\rho_0,X,\alpha}
        \left(\rho_0+\frac{\lambda}{Y}\right)
        =
        O_K\left(e^{-c_KY}\right)
        \label{eq:remainder-bound}
\end{equation}
uniformly for \(\lambda\in K\). Consequently,
\begin{equation}
        \Delta_{X,\alpha}
        \left(\rho_0+\frac{\lambda}{Y}\right)
        =
        -a\,
        \Ree\left(\frac{e^{-\lambda}}{\lambda}\right)
        +
        O_K\left(\frac1Y+e^{-c_KY}\right).
        \label{eq:full-profile}
\end{equation}
\end{theorem}

The hypotheses have separate roles. The finite-window condition is a
zero-configuration condition: it is checked by evaluating
\(Q_\alpha(\rho';\rho_0)\) for non-selected crossed zeros in the window
\(|\gamma'-\gamma_0|\le M\). The pole condition controls the residue at
\(s+z=1\), while the shifted-contour hypothesis controls only the new vertical
line. Together, damping of the non-selected residues and regularity of the
shifted contour yield the selected-zero profile as the full local profile
up to an exponentially small error.

\begin{proof}
The proof is given in Appendix~\ref{app:conditional-localization}. It estimates
the four parts of the nonlocal remainder in Eq.~\eqref{eq:nonlocal-remainder}:
the crossed non-selected zeros, the pole term, the possible trivial-zero
terms, and the shifted contour integral. These estimates give
Eq.~\eqref{eq:remainder-bound}, and Theorem~\ref{thm:central-profile} then
gives Eq.~\eqref{eq:full-profile}.
\end{proof}

The damping condition in Theorem~\ref{thm:full-remainder-localization}
reflects the anisotropy of the Gaussian--Perron factor. 
Proposition~\ref{prop:damping-geometry} gives the boundary
\(Q_\alpha=0\): vertical displacement damps a non-selected zero, while
horizontal displacement can increase its exponential weight. 
The theorem applies when the damping and shifted-contour regularity
hypotheses both hold.

The following Corollary is a direct
consequence of the preceding theorem.

\begin{corollary}[Conditional full local profile]
\label{cor:conditional-full-profile}
Under the hypotheses of Theorem~\ref{thm:full-remainder-localization},
\begin{equation}
        \Delta_{X,\alpha}
        \left(\rho_0+\frac{\lambda}{\log X}\right)
        =
        -a\,
        \Ree\left(\frac{e^{-\lambda}}{\lambda}\right)
        +O_K\left(\frac1{\log X}\right)
\end{equation}
uniformly for \(\lambda=a+ib\in K\).
\end{corollary}

\begin{proof}
This follows from Eq.~\eqref{eq:full-profile} and \(Y=\log X\).
\end{proof}

We may now formulate:

\begin{theorem}[RH-conditional full local profile]
\label{thm:rh-full-local-profile}
Assume RH. Let \(\rho_0=1/2+i\gamma_0\), \(\gamma_0>0\), be a simple
non-trivial zero of \(\zeta\), let \(K\Subset\C\setminus\{0\}\), and put
\begin{equation}
        s=\rho_0+\frac{\lambda}{Y},
        \qquad
        \lambda=a+ib\in K,
        \qquad
        Y=\log X.
\end{equation}
Assume
\begin{equation}
        \frac12+\alpha^2\left(\frac14-\gamma_0^2\right)<0.
        \label{eq:rh-pole-damping}
\end{equation}
Choose
\begin{equation}
        c>\frac12,
        \qquad
        0<d<\min\left\{\frac12,\frac1{\alpha^2}\right\}.
\end{equation}
Then, uniformly for \(\lambda\in K\),
\begin{equation}
        \mathcal E_{\rho_0,X,\alpha}
        \left(\rho_0+\frac{\lambda}{Y}\right)
        =
        O_K\left(e^{-c_KY}\right)
        \label{eq:rh-remainder-bound}
\end{equation}
for some \(c_K>0\), with constants depending on \(K,\alpha,d,\rho_0\). Hence
\begin{equation}
        \Delta_{X,\alpha}
        \left(\rho_0+\frac{\lambda}{Y}\right)
        =
        -a\,
        \Ree\left(\frac{e^{-\lambda}}{\lambda}\right)
        +
        O_K\left(\frac1Y+e^{-c_KY}\right),
        \label{eq:rh-full-profile}
\end{equation}
or, equivalently,
\begin{equation}
        \Delta_{X,\alpha}
        \left(\rho_0+\frac{\lambda}{\log X}\right)
        =
        -a\,
        \Ree\left(\frac{e^{-\lambda}}{\lambda}\right)
        +
        O_K\left(\frac1{\log X}+X^{-c_K}\right).
\end{equation}
\end{theorem}

The proof is given in Appendix~\ref{app:rh-localization}. The theorem gives an
RH-conditional local law, uniform for \(\lambda\) in compact subsets of
\(\C\setminus\{0\}\), for each fixed zero \(\rho_0\). Under pole damping,
the selected residue determines the full defect asymptotics.

Under RH, hypothesis~\textup{(i)} of
Theorem~\ref{thm:full-remainder-localization} is automatic: every non-trivial
zero satisfies \(x_{\rho'}=0\), so
\(Q_\alpha(\rho';\rho_0)=-\alpha^2 y_{\rho'}^2\le 0\), with strict inequality
for \(\rho'\neq\rho_0\). Theorem~\ref{thm:rh-full-local-profile} is therefore
an application of Theorem~\ref{thm:full-remainder-localization} in which the
finite-window damping is supplied for free by the vertical alignment of
zeros; its substantive content is the residue asymptotics of
Theorem~\ref{thm:central-profile} together with the pole-damping and
shifted-contour hypotheses.

\begin{remark}
The pole-damping condition Eq.~\eqref{eq:rh-pole-damping} is equivalent to
\begin{equation}
        \alpha>
        \frac{1}{\sqrt{2(\gamma_0^2-1/4)}}.
\end{equation}
Thus it is an explicit lower bound on the smoothing width for the fixed zero
\(\rho_0\). For the classical non-trivial zeros of \(\zeta\), this restriction
is mild for fixed moderate \(\alpha\), such as \(\alpha=1\). 
The condition imposes enough smoothing width to give exponential suppression
of the pole contribution. In particular, \(\alpha\ge1\) suffices whenever
\(\gamma_0^2>3/4\), while smaller values of \(\alpha\) require the displayed
lower bound.
\end{remark}

\section{Direct Prime-Side Numerical Check}
\label{sec:computational-diagnostics}

The results in this paper are analytic. We include a direct numerical
check of the prime-side defect in order to illustrate the selected-zero
local profile in a concrete computable regime.

Let
\begin{equation}
        \rho_0=\frac12+14.1347251417347\,i
\end{equation}
be the first non-trivial zero of \(\zeta\), and put
\begin{equation}
        s_\lambda=\rho_0+\frac{\lambda}{\log X}.
\end{equation}
For real positive \(\lambda\), the finite-\(X\) selected-zero prediction
from Theorem~\ref{thm:central-profile} becomes
\begin{equation}
        \mathcal L_{\rho_0,X,\alpha}(s_\lambda)
        =
        -\lambda\,
        \Ree\left(
        \frac{\exp(-\lambda+\alpha^2\lambda^2/\log X)}{\lambda}
        \right)
        =
        -\exp\left(
        -\lambda+\frac{\alpha^2\lambda^2}{\log X}
        \right).
\end{equation}
We compare this expression with the directly computed prime-side defect
\begin{equation}
        \Delta_{X,\alpha}(s_\lambda)
        =
        \left(\sigma-\frac12\right)
        \Ree\left(
        P_{X,\alpha}(s_\lambda)-\frac{\zeta'}{\zeta}(s_\lambda)
        \right),
\end{equation}
where \(P_{X,\alpha}\) is evaluated from the truncated prime-power sum
\begin{equation}
        P_{X,\alpha}^{(N)}(s)
        =
        -\sum_{\substack{n=p^k\\ n\le N}}
        \Lambda(n)n^{-s}
        \frac12
        \operatorname{erfc}
        \left(
        \frac{\log n-\log X}{2\alpha\sqrt{\log X}}
        \right).
\end{equation}
The truncation is numerical only; the analytic object remains the
infinite smoothed prime-power field.

Figure~\ref{fig:prime-side-first-zero-overlay} shows the comparison for
\begin{equation}
        X=10^4,\qquad \alpha=0.2,\qquad
        N=2{,}000{,}000,\qquad
        0.25\le \lambda\le 3.5 .
\end{equation}
The measured prime-side defect is visually indistinguishable from the
finite-\(X\) selected-zero profile. The limiting profile
\(-e^{-\lambda}\) is also shown for comparison; its slight displacement
indicates the size of the finite-\(X\) correction.

\begin{figure}[htbp]
\centering
\includegraphics[width=0.82\textwidth]{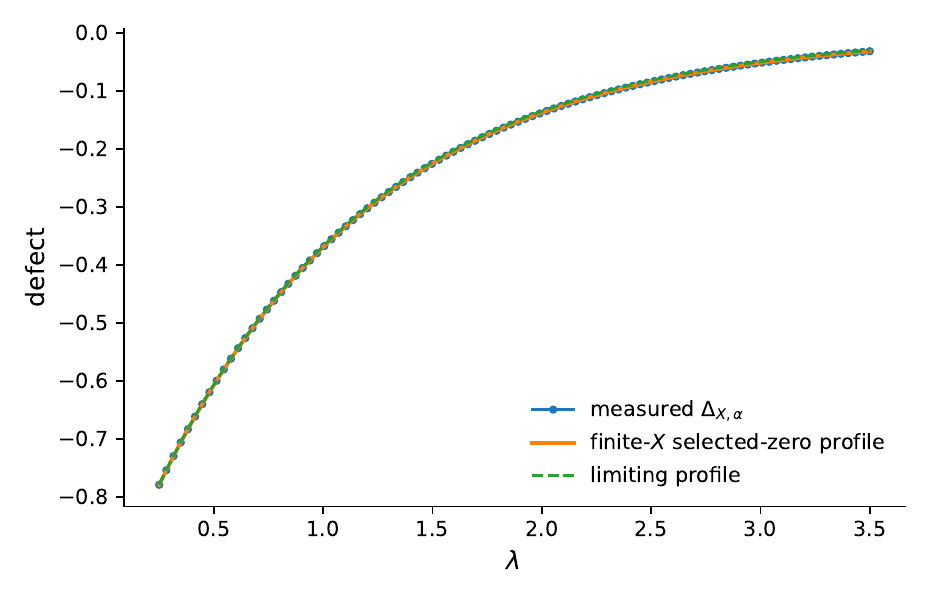}
\caption{Direct prime-side numerical check near the first non-trivial zero
\(\rho_0=1/2+14.1347251417347\,i\). The measured defect
\(\Delta_{X,\alpha}(s_\lambda)\), computed from the truncated smoothed
prime-power sum, is compared with the finite-\(X\) selected-zero prediction
\(\mathcal L_{\rho_0,X,\alpha}(s_\lambda)\). The limiting profile is shown
only as a reference. Here \(X=10^4\), \(\alpha=0.2\),
\(N=2{,}000{,}000\), and \(0.25\le\lambda\le3.5\).}
\label{fig:prime-side-first-zero-overlay}
\end{figure}

For this run, the maximum observed pointwise discrepancy is
\begin{equation}
        \max_{0.25\le\lambda\le3.5}
        \left|
        \Delta_{X,\alpha}(s_\lambda)
        -
        \mathcal L_{\rho_0,X,\alpha}(s_\lambda)
        \right|
        =
        3.40\times10^{-10}.
\end{equation}
Thus, in this controlled parameter regime, the direct prime-side
computation reproduces the finite-\(X\) selected-zero profile to very high
accuracy.

As a small robustness check, Table~\ref{tab:prime-side-robustness} reports
the same maximum discrepancy for several values of \(X\), \(\alpha\), and
the prime-power cutoff \(N\). The residual remains small throughout the
tested range.

\begin{table}[htbp]
\centering
\caption{Robustness check for the direct prime-side computation near
\(\rho_0=1/2+14.1347251417347\,i\). The final column reports
\(\max|\Delta_{X,\alpha}-\mathcal L_{\rho_0,X,\alpha}|\) over
\(0.25\le\lambda\le3.5\).}
\label{tab:prime-side-robustness}
\begin{tabular}{c c r c}
\hline
\(X\) & \(\alpha\) & \(N\) &
\(\max|\Delta_{X,\alpha}-\mathcal L_{\rho_0,X,\alpha}|\) \\
\hline
\(10^3\) & \(0.15\) & \(300{,}000\) & \(1.13\times10^{-6}\) \\
\(10^3\) & \(0.20\) & \(500{,}000\) & \(7.72\times10^{-9}\) \\
\(10^3\) & \(0.25\) & \(1{,}000{,}000\) & \(7.46\times10^{-12}\) \\
\(10^4\) & \(0.15\) & \(1{,}000{,}000\) & \(1.11\times10^{-7}\) \\
\(10^4\) & \(0.20\) & \(2{,}000{,}000\) & \(3.40\times10^{-10}\) \\
\(10^4\) & \(0.25\) & \(8{,}000{,}000\) & \(9.56\times10^{-10}\) \\
\hline
\end{tabular}
\end{table}

This computation is not used in the proof. Its purpose is to show that the
smoothed prime-power field gives a numerically accessible realization of
the local selected-zero profile, once the nonlocal remainder is small in
the tested window.

\section{Possible Extensions}

\subsection{Sharper zero-spacing hypotheses}

The finite-window margin \(\eta_{\alpha,M}(\rho_0)\) suggests a quantitative
zero-spacing problem adapted to the anisotropic Gaussian--Perron kernel.
Future refinements could control this margin uniformly for zeros up to height
\(T\), with explicit dependence on local gaps, or on average for almost all
zeros.

\subsection{\texorpdfstring{Dirichlet \(L\)-functions}{Dirichlet L-functions}}

Another direction is to adapt the construction to primitive Dirichlet
\(L\)-functions by replacing \(\zeta'/\zeta\) with \(L'/L(s,\chi)\). The
completed functional equation would then determine the relevant central line
and the analog of the horizontal force.

\subsection{Other kernels}

The Gaussian--Perron kernel is useful because it gives both a clean contour
shift and an explicit error-function prime weight. 
Other admissible kernels lead to different selected-zero profiles and damping
regions, making kernel stability a natural extension.

\section{Conclusion}

We introduced the Gaussian--Perron prime-force defect
\begin{equation}
        \Delta_{X,\alpha}(s)
        =
        \left(\sigma-\frac12\right)
        \Ree\left(P_{X,\alpha}(s)-\frac{\zeta'}{\zeta}(s)\right),
\end{equation}
where \(P_{X,\alpha}\) is an explicit smoothed prime-side field. A contour
shift gives the corresponding zero-side formula.

Near a simple zero \(\rho_0\), on the scale
\(s=\rho_0+\lambda/\log X\), the selected zero contributes
\begin{equation}
        -a\,
        \Ree\left(\frac{e^{-\lambda}}{\lambda}\right)
        +O\left(\frac1{\log X}\right)
\end{equation}
when \(\rho_0\) lies on the critical line. A hypothetical off-line zero gives
a selected-zero term with linear growth in \(\log X\). Under explicit
damping-separation, pole-damping, and shifted-contour regularity hypotheses,
the nonlocal remainder is exponentially small and the full defect has the
selected-zero profile. Under RH and the pole-damping condition, this becomes a
full local profile theorem for each fixed simple critical-line zero.

The result is a local Gaussian-smoothed prime-force defect theory in which
the logarithmic selected-zero scale and the anisotropic damping boundary work
together. The Gaussian--Perron kernel supplies both the prime-side cutoff and
the zero-side localization mechanism.

Figure~\ref{fig:prime-side-first-zero-overlay} gives a direct prime-side
numerical check of the selected-zero local profile. Appendix~\ref{app:selected-zero-profiles}
records an additional residue-level surface visualization of the selected-zero
profiles.

\appendix

\part*{Appendix}
\setcounter{section}{0}
\setcounter{equation}{0}
\setcounter{figure}{0}
\setcounter{table}{0}
\renewcommand{\theequation}{A\arabic{equation}}
\renewcommand{\thefigure}{A\arabic{figure}}
\renewcommand{\thetable}{A\arabic{table}}
\renewcommand{\thesection}{A\arabic{section}}
\renewcommand{\thesubsection}{\thesection.\arabic{subsection}}
\renewcommand{\theHequation}{app.eq.\arabic{equation}}
\renewcommand{\theHfigure}{app.fig.\arabic{figure}}
\renewcommand{\theHtable}{app.tab.\arabic{table}}
\renewcommand{\theHsection}{app.sec.\arabic{section}}
\renewcommand{\theHsubsection}{app.subsec.\arabic{section}.\arabic{subsection}}

\addcontentsline{toc}{section}{Appendix}
\label{app:Appendix}

\section{Proof of the Conditional Full Nonlocal Remainder Localization Theorem}
\label{app:conditional-localization}

This section proves Theorem~\ref{thm:full-remainder-localization}.
The proof follows the decomposition
\begin{equation}
        \Delta_{X,\alpha}(s)
        =
        \mathcal L_{\rho_0,X,\alpha}(s)
        +
        \mathcal E_{\rho_0,X,\alpha}(s),
\end{equation}
and estimates each component of the nonlocal remainder separately.

\subsection{Auxiliary zero-counting and Gaussian-tail estimates}

The proof uses two standard estimates for the large-height part of the
non-selected zero contribution.

\begin{lemma}[Zero-counting input]
\label{lem:zero-counting-input}
For \(T\ge2\) and \(U\ge1\), the number of non-trivial zeros with
\begin{equation}
        |\gamma'-T|\le U
\end{equation}
is
\begin{equation}
        O\bigl((U+1)\log(T+U+2)\bigr).
\end{equation}
In particular, for fixed \(U\), this number is \(O_U(\log(T+2))\). This
follows from the Riemann--von Mangoldt formula; see
\cite{Titchmarsh,MontgomeryVaughan}.
\end{lemma}

\begin{lemma}[Large-height Gaussian tail]
\label{lem:large-height-tail}
Let \(\rho_0=1/2+i\gamma_0\), let \(K\Subset\C\setminus\{0\}\), and let
\(\alpha>0\). There exist constants \(M_0=M_0(\alpha,K)\),
\(Y_0=Y_0(\alpha,K)\), and \(c_{\rm tail}=c_{\rm tail}(\alpha,K)>0\) such
that, for all \(Y\ge Y_0\), all \(\lambda\in K\), and every non-trivial zero
\(\rho'\), the following holds. If
\begin{equation}
        \delta_{\rho'}=\rho'-\rho_0=x_{\rho'}+iy_{\rho'},
        \qquad
        |y_{\rho'}|\ge M_0,
\end{equation}
then
\begin{equation}
        Y\Ree\left(\delta_{\rho'}-\frac{\lambda}{Y}\right)
        +
        \alpha^2Y
        \Ree\left(
        \left(\delta_{\rho'}-\frac{\lambda}{Y}\right)^2
        \right)
        \le
        -c_{\rm tail}Yy_{\rho'}^2 .
\end{equation}
\end{lemma}

\begin{proof}
Since \(\rho'\) lies in the critical strip and \(\rho_0=1/2+i\gamma_0\),
\(|x_{\rho'}|\le1/2\). Also \(|\lambda|\le C_K\) for \(\lambda\in K\).
Expanding gives
\begin{equation}
\begin{aligned}
&Y\Ree\left(\delta_{\rho'}-\frac{\lambda}{Y}\right)
+
\alpha^2Y
\Ree\left(
\left(\delta_{\rho'}-\frac{\lambda}{Y}\right)^2
\right)       \\
&
=Yx_{\rho'}-\Ree\lambda
+
\alpha^2Y(x_{\rho'}^2-y_{\rho'}^2)
-
2\alpha^2\Ree(\delta_{\rho'}\lambda)
+
\frac{\alpha^2}{Y}\Ree(\lambda^2)                         \\
&
\le
-\alpha^2Yy_{\rho'}^2
+
Y\left(\frac12+\frac{\alpha^2}{4}\right)
+
C_{\alpha,K}(1+|y_{\rho'}|).
\end{aligned}
\end{equation}
Choose \(M_0\ge1\) so that
\begin{equation}
        \frac12+\frac{\alpha^2}{4}
        \le
        \frac{\alpha^2}{4}y_{\rho'}^2
        \qquad(|y_{\rho'}|\ge M_0),
\end{equation}
and then choose \(Y_0\) so that, for \(Y\ge Y_0\),
\begin{equation}
        C_{\alpha,K}(1+|y_{\rho'}|)
        \le
        \frac{\alpha^2}{4}Yy_{\rho'}^2
        \qquad(|y_{\rho'}|\ge M_0).
\end{equation}
The preceding bound then gives
\begin{equation}
        Y\Ree\left(\delta_{\rho'}-\frac{\lambda}{Y}\right)
        +
        \alpha^2Y
        \Ree\left(
        \left(\delta_{\rho'}-\frac{\lambda}{Y}\right)^2
        \right)
        \le
        -\frac{\alpha^2}{2}Yy_{\rho'}^2.
\end{equation}
Thus the result holds with \(c_{\rm tail}=\alpha^2/2\), after increasing
\(M_0\) and \(Y_0\) if necessary.
\end{proof}

\subsection{Proof of the theorem}

\begin{proof}[Proof of Theorem~\ref{thm:full-remainder-localization}]
By Eq.~\eqref{eq:full-decomposition},
\begin{equation}
        \Delta_{X,\alpha}(s)
        =
        \mathcal L_{\rho_0,X,\alpha}(s)
        +
        \mathcal E_{\rho_0,X,\alpha}(s).
\end{equation}
The selected-zero term is given by Theorem~\ref{thm:central-profile}; it
remains to estimate the nonlocal remainder.

Since \(K\) is compact, for all sufficiently large \(Y\),
\begin{equation}
        -d<-\frac{\Ree\lambda}{Y}<c
\end{equation}
uniformly for \(\lambda\in K\). Hence \(\rho_0\) is crossed, and
\(\Ree(s+c)>1\). Thus Proposition~\ref{prop:explicit-remainder} applies.

Let \(\rho'=\beta'+i\gamma'\neq\rho_0\) be a crossed non-trivial zero and put
\begin{equation}
        \delta_{\rho'}=\rho'-\rho_0=x_{\rho'}+iy_{\rho'},
        \qquad
        y_{\rho'}=\gamma'-\gamma_0 .
\end{equation}
Since \(s=\rho_0+\lambda/Y\),
\begin{equation}
        \rho'-s=\delta_{\rho'}-\frac{\lambda}{Y}.
\end{equation}
The modulus of the exponential factor in the zero contribution is
\begin{equation}
        \left|
        X^{\rho'-s}\exp\!\left(\alpha^2Y(\rho'-s)^2\right)
        \right|
        =
        \exp\!\left(
        Y\Ree(\rho'-s)+\alpha^2Y\Ree\left((\rho'-s)^2\right)
        \right).
\end{equation}

First consider the crossed zeros with \(|y_{\rho'}|\le M\). They form a
finite set, and hypothesis \textup{(i)} gives a damping margin at least
\(\kappa\). Proposition~\ref{prop:finite-window-certificate} therefore gives
their total contribution as
\begin{equation}
        O_{K,M,\rho_0,\alpha}(e^{-c_1Y})
\end{equation}
for some \(c_1>0\).

Now consider the tail \(|y_{\rho'}|>M\). Since \(M\ge M_0(\alpha,K)\),
Lemma~\ref{lem:large-height-tail} gives
\begin{equation}
        Y\Ree(\rho'-s)+\alpha^2Y\Ree\left((\rho'-s)^2\right)
        \le
        -c_{\rm tail}Yy_{\rho'}^2 .
\end{equation}
After increasing \(M\) if necessary, assume \(M\ge1\). Since
\(|\lambda|\le C_K\) on \(K\), for all sufficiently large \(Y\),
\begin{equation}
\begin{aligned}
        |\rho'-s|
        &=
        \left|\delta_{\rho'}-\frac{\lambda}{Y}\right|        \\
        &\ge
        |y_{\rho'}|-\frac{C_K}{Y}                            \\
        &\ge
        \frac12|y_{\rho'}|
        \ge
        \frac14(1+|y_{\rho'}|).
\end{aligned}
\end{equation}
Thus each tail zero contributes at most
\begin{equation}
        \frac{C_K}{1+|y_{\rho'}|}
        \exp\left(-c_{\rm tail}Yy_{\rho'}^2\right).
\end{equation}
Using Lemma~\ref{lem:zero-counting-input} in unit intervals gives
\begin{equation}
        \sum_{m\ge M}
        (m+1)\log(\gamma_0+m+3)
        \exp(-c_{\rm tail}Ym^2)
        =
        O_{K,M,\rho_0}(e^{-c_3Y}),
\end{equation}
for some \(c_3>0\); for instance one may take
\(c_3=\frac12c_{\rm tail}M^2\) after absorbing polynomial and logarithmic
factors. Hence all crossed non-selected zeros contribute
\(O_K(e^{-c_{\rm zero}Y})\) for some \(c_{\rm zero}>0\).

The pole term at \(1\) is
\begin{equation}
        -
        \left(\sigma-\frac12\right)
        \Ree
        \frac{X^{1-s}\exp(\alpha^2Y(1-s)^2)}{1-s}.
\end{equation}
Its exponential factor has modulus
\begin{equation}
        \exp\left(
        Y\Ree(1-s)+\alpha^2Y\Ree\left((1-s)^2\right)
        \right).
\end{equation}
By Eq.~\eqref{eq:pole-damping-condition}, this is
\(O_K(e^{-\eta_{\rm pole}Y})\). Since
\(\sigma-1/2=O_K(1/Y)\) and \(1-s\) is bounded away from zero, the pole
contribution is \(O_K(e^{-c_{\rm pole}Y})\).

By hypothesis \textup{(iii)},
\begin{equation}
        d<2+\Ree\left(\rho_0+\frac{\lambda}{Y}\right)
        =
        2+\Ree s
\end{equation}
for all sufficiently large \(Y\), uniformly for \(\lambda\in K\). Hence no
trivial zeros are crossed, by Proposition~\ref{prop:trivial-zero}.

Finally consider the shifted contour term. Under hypothesis \textup{(iv)}, the
admissible-height limit defining \(R_{X,\alpha,d}(s)\) is dominated by the
vertical-line majorant below and may be estimated over \(\Ree z=-d\). On
\(z=-d+iv\),
\begin{equation}
        \left|X^z\exp(\alpha^2Yz^2)\right|
        =
        \exp\!\left(-Y(d-\alpha^2d^2)-\alpha^2Yv^2\right).
\end{equation}
Since \(d<1/\alpha^2\), \(d-\alpha^2d^2>0\). By the assumed
logarithmic-derivative bound on the shifted line,
\begin{equation}
        \frac{\zeta'}{\zeta}(s-d+iv)
        \ll_{d,K,\rho_0,A}
        \log^A(2+|v|+\gamma_0).
\end{equation}
Therefore
\begin{equation}
\begin{aligned}
        |R_{X,\alpha,d}(s)|
        &\le
        C_{d,K,\rho_0}
        e^{-Y(d-\alpha^2d^2)}
        \int_{-\infty}^{\infty}
        \log^A(2+|v|+\gamma_0)
        e^{-\alpha^2Yv^2}
        \frac{\dd v}{|-d+iv|}.
\end{aligned}
\end{equation}
Since \(d>0\), \(|-d+iv|\ge d\). Denoting the preceding integral by
\(J_Y\),
\begin{equation}
\begin{aligned}
        J_Y
        &\ll_{d,\gamma_0,A}
        \int_{-\infty}^{\infty}
        \log^A(2+|v|+\gamma_0)e^{-\alpha^2Yv^2}\,\dd v        \\
        &\ll_{d,\gamma_0,\alpha,A}
        Y^{-1/2}\log^A(3+\gamma_0).
\end{aligned}
\end{equation}
Therefore, since \(d-\alpha^2d^2>0\),
\begin{equation}
\begin{aligned}
        |R_{X,\alpha,d}(s)|
        &\ll_{d,K,\rho_0,\alpha,A}
        Y^{-1/2}e^{-Y(d-\alpha^2d^2)}        \\
        &=O_{d,K,\rho_0,\alpha}(e^{-c_RY})
\end{aligned}
\end{equation}
for some \(c_R>0\). The factor \(\sigma-1/2=O_K(1/Y)\) further improves the
bound.

Combining the zero, pole, trivial-zero, and shifted-contour estimates gives
Eq.~\eqref{eq:remainder-bound}. Theorem~\ref{thm:central-profile} then gives
Eq.~\eqref{eq:full-profile}.
\end{proof}

\section{Proof of the RH-Conditional Full Local Profile}
\label{app:rh-localization}

Throughout the appendix we assume RH and prove
Theorem~\ref{thm:rh-full-local-profile}. The zero
\(\rho_0=1/2+i\gamma_0\) is fixed and simple, and
\begin{equation}
        s=\rho_0+\frac{\lambda}{Y},
        \qquad
        \lambda=a+ib\in K,
        \qquad
        Y=\log X,
\end{equation}
where \(K\Subset\C\setminus\{0\}\). All constants depend at most 
on \(K,\alpha,d,\rho_0\) and are uniform in
\(X\) and \(\lambda\).

\begin{lemma}[Gaussian localization of the non-selected zeros under RH]
\label{lem:rh-zero-localization}
Assume RH, and let \(\rho_0=1/2+i\gamma_0\) be a simple zero. Then there are
constants \(C_K,c_K>0\) such that, uniformly for \(\lambda\in K\),
\begin{equation}
        \sum_{\rho'\neq\rho_0}
        \frac{
        m_{\rho'}
        \left|
        X^{\rho'-s}
        \exp\left(\alpha^2Y(\rho'-s)^2\right)
        \right|
        }{
        |\rho'-s|
        }
        \le C_K e^{-c_KY}.
        \label{eq:rh-zero-localization}
\end{equation}
\end{lemma}

\begin{proof}
Under RH, every non-trivial zero is \(\rho'=1/2+i\gamma'\). Put
\(y_{\rho'}=\gamma'-\gamma_0\). Since
\begin{equation}
        \rho'-s
        =
        -\frac{a}{Y}
        +
        i\left(y_{\rho'}-\frac{b}{Y}\right),
\end{equation}
we have
\begin{equation}
\begin{aligned}
        \left|
        X^{\rho'-s}
        \exp\left(\alpha^2Y(\rho'-s)^2\right)
        \right|
        &=
        \exp\left(
        -a+\frac{\alpha^2a^2}{Y}
        -
        \alpha^2Y\left(y_{\rho'}-\frac{b}{Y}\right)^2
        \right)                                        \\
        &\le
        C_K
        \exp\left(
        -\alpha^2Y\left(y_{\rho'}-\frac{b}{Y}\right)^2
        \right).
\end{aligned}
        \label{eq:rh-zero-gaussian-bound}
\end{equation}

Since \(\rho_0\) is fixed and simple, and the zero set is discrete, there is a
nearest-neighbor gap
\begin{equation}
        g_0=\inf_{\rho'\neq\rho_0}|\gamma'-\gamma_0|>0.
\end{equation}
Hence, for all large \(Y\), uniformly for \(\lambda\in K\),
\begin{equation}
        \left|y_{\rho'}-\frac{b}{Y}\right|\ge\frac{g_0}{2}
        \qquad(\rho'\neq\rho_0).
\end{equation}
Thus the zeros in any fixed window \(|y_{\rho'}|\le M\) contribute
\(O_K(e^{-c_1Y})\): there are finitely many of them, and
\begin{equation}
        |\rho'-s|\ge |y_{\rho'}-b/Y|\ge g_0/2 .
\end{equation}

It remains to estimate the tail. Choose \(M\ge1\) so large that, for
\(|y_{\rho'}|\ge M\) and all large \(Y\),
\begin{equation}
        \left|y_{\rho'}-\frac{b}{Y}\right|\ge\frac{|y_{\rho'}|}{2}.
\end{equation}
Since \(|\lambda|\le C_K\) on \(K\), increasing \(Y\) if necessary gives
\begin{equation}
        |\rho'-s|
        =
        \left|iy_{\rho'}-\frac{\lambda}{Y}\right|
        \ge
        \frac12|y_{\rho'}|
        \ge
        \frac14(1+|y_{\rho'}|).
\end{equation}
Together with Eq.~\eqref{eq:rh-zero-gaussian-bound}, this gives
\begin{equation}
        \exp\left(
        -\alpha^2Y
        \left(y_{\rho'}-\frac{b}{Y}\right)^2
        \right)
        \le
        \exp\left(-\frac{\alpha^2}{4}Yy_{\rho'}^2\right).
\end{equation}
Using Lemma~\ref{lem:zero-counting-input} in unit intervals
\(m\le |y_{\rho'}|<m+1\), and absorbing finitely many low-ordinate intervals
into the constant, we obtain
\begin{equation}
\begin{aligned}
&\sum_{|y_{\rho'}|\ge M}
        \frac{
        m_{\rho'}
        \left|
        X^{\rho'-s}
        \exp\left(\alpha^2Y(\rho'-s)^2\right)
        \right|
        }{
        |\rho'-s|
        }                                               \\
&\qquad\ll_K
        \sum_{m\ge M}
        \log(\gamma_0+m+3)
        \exp\left(-\frac{\alpha^2}{4}Ym^2\right)
        =
        O_K(e^{-c_2Y})
\end{aligned}
\end{equation}
for some \(c_2>0\). Combining the bounded-window and tail bounds proves
Eq.~\eqref{eq:rh-zero-localization}.
\end{proof}

The next two lemmas control the pole, trivial-zero, and shifted-contour
terms in the proof of Theorem~\ref{thm:rh-full-local-profile}.

\begin{lemma}[Pole and trivial-zero terms under the RH-local hypotheses]
\label{lem:rh-pole-trivial}
Assume the hypotheses of Theorem~\ref{thm:rh-full-local-profile}. 
Then the pole contribution in Eq.~\eqref{eq:nonlocal-remainder} is \(O_K(e^{-cY})\),
and the trivial-zero contribution vanishes for all sufficiently large \(Y\).
\end{lemma}

\begin{proof}
The pole-damping expansion used in Theorem~\ref{thm:full-remainder-localization}
gives, uniformly for \(\lambda\in K\),
\begin{equation}
        \Ree(1-s)
        +
        \alpha^2\Ree\left((1-s)^2\right)
        =
        \frac12+\alpha^2\left(\frac14-\gamma_0^2\right)
        +O_K\left(\frac1Y\right).
\end{equation}
By Eq.~\eqref{eq:rh-pole-damping}, the right-hand side is at most \(-\eta\)
for some \(\eta>0\) and all sufficiently large \(Y\). Since \(1-s\) stays
bounded away from zero and \(\sigma-1/2=a/Y=O_K(1/Y)\), the pole contribution
is \(O_K(e^{-cY})\).

For \(m\ge1\), the inequality \(0<d<1/2\) places each trivial-zero residue
to the left of the shifted contour for all sufficiently large \(Y\), uniformly
for \(\lambda\in K\). Hence the trivial-zero contribution vanishes.
\end{proof}

The following lemma is also used in the proof of
Theorem~\ref{thm:rh-full-local-profile}.

\begin{lemma}[Shifted-contour decay under RH]
\label{lem:rh-contour-decay}
Assume RH and choose
\begin{equation}
        0<d<\min\left\{\frac12,\frac1{\alpha^2}\right\}.
\end{equation}
Then
\begin{equation}
        R_{X,\alpha,d}(s)=O_K(e^{-cY})
        \label{eq:rh-contour-decay}
\end{equation}
uniformly for \(\lambda\in K\).
\end{lemma}

\begin{proof}
On \(z=-d+iv\),
\begin{equation}
        \left|
        X^z\exp(\alpha^2Yz^2)
        \right|
        =
        \exp\left(-Y(d-\alpha^2d^2)-\alpha^2Yv^2\right),
\end{equation}
and \(d-\alpha^2d^2>0\). For \(w=s+z\),
\begin{equation}
        \Ree w=\frac12-d+\frac{a}{Y}.
\end{equation}
Since \(K\) is compact and \(d<1/2\), these lines lie, for large \(Y\), in a
fixed compact substrip of \(0<\sigma<1/2\), separated from the critical line.
Under RH they are therefore separated from the non-trivial zeros, and they are
also separated from the pole at \(w=1\) and the trivial zeros.

The fixed separation from zeros and the pole allows
Lemma~\ref{lem:log-derivative-away} to be applied. Hence
\begin{equation}
        \frac{\zeta'}{\zeta}(w)
        \ll_{d,K,\rho_0}\log^2(3+|\Imm w|),
\end{equation}
uniformly for \(\lambda\in K\).

The same Gaussian vertical-line estimate used in the proof of
Theorem~\ref{thm:full-remainder-localization}, now with the above
logarithmic-derivative bound, gives
\begin{equation}
        |R_{X,\alpha,d}(s)|
        \ll_{d,K,\rho_0,\alpha}
        Y^{-1/2}e^{-Y(d-\alpha^2d^2)}.
\end{equation}
Hence
\begin{equation}
        |R_{X,\alpha,d}(s)|
        \ll_{d,K,\rho_0,\alpha}
        Y^{-1/2}e^{-Y(d-\alpha^2d^2)}
        =
        O_K(e^{-cY})
\end{equation}
for some \(c>0\). This proves Eq.~\eqref{eq:rh-contour-decay}.
\end{proof}

\begin{proof}[Proof of Theorem~\ref{thm:rh-full-local-profile}]
For large \(Y\), compactness of \(K\) gives
\begin{equation}
        -d<-\frac{\Ree\lambda}{Y}<c
\end{equation}
uniformly for \(\lambda\in K\). Under RH, every non-trivial zero therefore
satisfies the crossed-zero condition in Theorem~\ref{thm:explicit-defect},
and \(\rho_0\) is crossed. Since \(c>1/2\), also \(\Ree(s+c)>1\). Thus
Proposition~\ref{prop:explicit-remainder} applies.

By Proposition~\ref{prop:explicit-remainder}, the remainder consists of the
non-selected zero sum, the pole term, the trivial-zero term, and the shifted
contour term, each multiplied by \(\sigma-1/2\) after taking real parts.
Lemma~\ref{lem:rh-zero-localization}, Lemma~\ref{lem:rh-pole-trivial}, and
Lemma~\ref{lem:rh-contour-decay} give respectively
\begin{equation}
        \mathcal E_{\rho_0,X,\alpha}
        \left(\rho_0+\frac{\lambda}{Y}\right)
        =
        O_K(e^{-c_KY}),
\end{equation}
after decreasing \(c_K>0\) if necessary. This proves
Eq.~\eqref{eq:rh-remainder-bound}.

Finally, Theorem~\ref{thm:central-profile} gives, uniformly for
\(\lambda\in K\),
\begin{equation}
        \mathcal L_{\rho_0,X,\alpha}
        \left(\rho_0+\frac{\lambda}{Y}\right)
        =
        -a\,\Ree\left(\frac{e^{-\lambda}}{\lambda}\right)
        +
        O_K\left(\frac1Y\right).
\end{equation}
Combining this with the remainder estimate proves
Eq.~\eqref{eq:rh-full-profile}. Since \(Y=\log X\), the final displayed form
follows.
\end{proof}

\section{Additional Selected-Zero Profile Visualization}
\label{app:selected-zero-profiles}

This appendix records the residue-level selected-zero surface visualization
associated with Theorems~\ref{thm:central-profile}
and~\ref{thm:offline-spike}. The figure displays only the isolated
selected-zero contribution
\(\mathcal L_{\rho_0,X,\alpha}\), not the full defect
\(\Delta_{X,\alpha}\). The nonlocal remainder
\(\mathcal E_{\rho_0,X,\alpha}\) is omitted.

The displayed quantity is
\begin{equation}
        \mathcal L_{\rho_0,X,\alpha}
        \left(\rho_0+\frac{\lambda}{Y}\right),
        \qquad
        \lambda=a+ib,
        \qquad
        Y=\log X.
\end{equation}
For a critical-line zero, this profile is
\begin{equation}
        \mathcal L_{\rho_0,X,\alpha}
        \left(\rho_0+\frac{\lambda}{Y}\right)
        =
        -a\,
        \Ree\left(
        \frac{\exp(-\lambda+\alpha^2\lambda^2/Y)}{\lambda}
        \right).
\end{equation}
For a hypothetical off-line zero with \(\beta_0-1/2\neq0\),
Theorem~\ref{thm:offline-spike} predicts a selected-zero contribution whose
leading scale is proportional to \(Y=\log X\).

\begin{figure}[htbp]
\centering
\includegraphics[width=\textwidth]{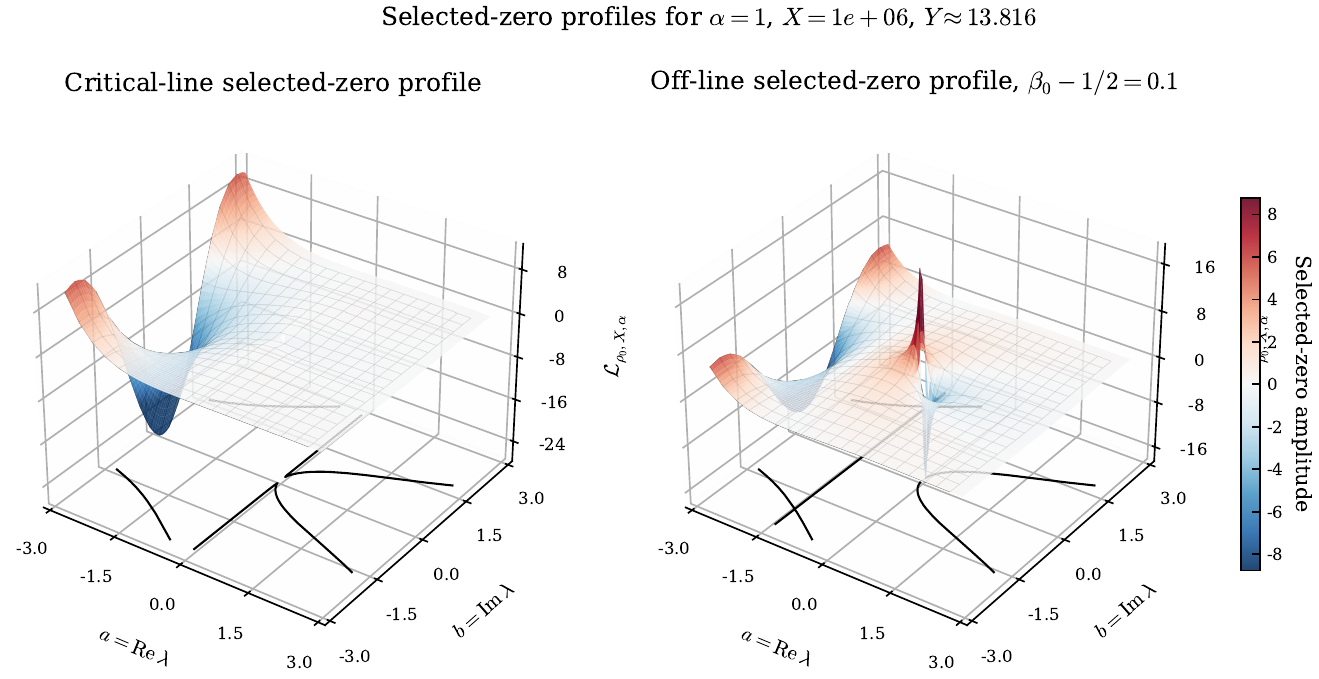}
\caption{Residue-level selected-zero profiles for the Gaussian--Perron
prime-force defect. The left panel shows the critical-line profile from
Theorem~\ref{thm:central-profile}, with \(\beta_0=1/2\). The right panel
shows the corresponding selected-zero profile for a hypothetical off-line
displacement \(\beta_0-1/2=0.10\), illustrating the spike mechanism in
Theorem~\ref{thm:offline-spike}. In both panels \(X=10^6\), \(Y=\log X\),
\(\alpha=1\), and \(\lambda=a+ib\). The disk around \(\lambda=0\) is masked
because the selected-zero term has a pole there. The displayed quantity is
\(\mathcal L_{\rho_0,X,\alpha}\); the full defect \(\Delta_{X,\alpha}\) is
obtained only after adding the nonlocal remainder
\(\mathcal E_{\rho_0,X,\alpha}\).}
\label{fig:selected-zero-profiles}
\end{figure}

\section*{Funding}
This research received no external funding.

\section*{Conflicts of Interest}
The author declares no conflicts of interest.

\section*{Data Availability}
No external datasets were used in this research. 
Key simulation scripts and post-processing code
for numerical validation are available from the corresponding author upon reasonable request.

\addcontentsline{toc}{section}{References}
\printbibliography

\end{document}